\documentclass[11pt]{amsproc}
 \usepackage[margin=1in]{geometry}
\usepackage{setspace,fullpage}
\usepackage{graphicx}
\usepackage[nice]{nicefrac}
\usepackage{amssymb}
\usepackage{multirow}
\usepackage{array}

\usepackage{amsmath,amsthm,amscd,amssymb, mathrsfs}

\usepackage{latexsym}

\numberwithin{equation}{section}

\theoremstyle{plain}
\newtheorem{theorem}{Theorem}[section]
\newtheorem{lemma}[theorem]{Lemma}

\theoremstyle{definition}

\theoremstyle{remark}

\newtheorem{case[theorem]}{Case}

\title[\parbox{14cm}{\centering{Restriction estimates for paraboloids \hspace{1in}}} \quad]{Finite field restriction estimates for the paraboloid in high even dimensions}
\author{ Alex Iosevich, Doowon Koh, and Mark Lewko}

\address{Department of Mathematics\\
 University of Rochester \\
 Rochester, NY USA}
\email{iosevich@math.rochester.edu}

\address{Department of Mathematics\\
Chungbuk National University \\
Cheongju, Chungbuk 28644 Korea}
\email{koh131@chungbuk.ac.kr}

\address{New York, NY USA}
\email{mlewko@gmail.com}

\thanks{
The second listed author was supported by Basic Science Research Program through the National Research Foundation of Korea(NRF) funded by the Ministry of Education, Science and Technology(NRF-2015R1A1A1A05001374)
}


\begin{document}

\begin{abstract} We prove that the finite field Fourier extension operator for the paraboloid is bounded from $L^2\to L^r$ for $r\geq \frac{2d+4}{d}$ in even dimensions $d\ge 8$, which is the optimal $L^2$ estimate. For $d=6$ we obtain the optimal range $r> \frac{2d+4}{d}=8/3$, apart from the endpoint. For $d=4$ we improve the prior range of $r>16/5=3.2$ to $r\geq 28/9=3.111\ldots$, compared to the conjectured range of $r\geq3$. The key new ingredient is improved additive energy estimates for subsets of the paraboloid. 
\end{abstract}

\maketitle
\section{Introduction} 
The Fourier restriction conjecture is one of central open problems in Euclidean harmonic analysis and lies at the crossroads of geometric measure theory, analytic number theory, arithmetic combinatorics, and dispersive PDEs, among other fields. Given a surface $V$ in $R^d$ with surface measure $d\sigma$, the conjecture seeks exponent pairs $(p,r)$ such that one has the inequality 
$$ \|(fd\sigma)^\vee\|_{L^r(R^d)} \le C \|f\|_{L^p(V,\,d\sigma)}$$
where $(\cdot)^\vee$ denotes the inverse Fourier transform. We refer the reader to \cite{Bo91, BCT06, Gu15, Gu16, St78,Ta03, Ta04}, for a more detailed discussion of this conjecture, its connections to other areas of mathematics, and partial progress.

In 2002, Mockenhaupt and Tao \cite{MT04} initiated the study of analogous Fourier restriction problems in the finite field setting. Their work primarily focused on the case of the paraboloid in low dimensions. Subsequently a number of authors have improved their results, extended and generalized their results to higher dimensions and more general surfaces/algebraic varieties. See \cite{IK08, IK09, IK10, Ko16, Le13, Le14, LL10, SZ16}.

The $d$-dimensional finite field paraboloid is defined as $P=\{(\underline{x},\underline{x}\cdot \underline{x}): \underline{x} \in F^{d-1}\}$. The Fourier extension operator for $P$ maps functions on $P$ to functions on $F^d$ and is defined as follows
$$ (fd\sigma)^{\vee}(x) := \frac{1}{|P|} \sum_{\xi \in P} f(\xi)e\left(x \cdot \xi \right).$$
The finite field restriction problem for the paraboloid seeks exponents pairs $(p,r)$ such that one has the inequality
$$ \|(fd\sigma)^\vee\|_{L^r(F^d)} \lesssim \|f\|_{L^p(P,\,d\sigma)}$$
with a constant independent of the size of the finite field. 
Here the $L^r(F^d)$ norm is with respect to the counting measure on $F^d$ and the $L^p(P,\,d\sigma)$ norm is with respect to the normalized counting measure on $P$ which assigns a mass of $|P|^{-1}$ to each point in $P$. We denote the best constant such that the above inequality holds as $R^{*}(p\to r)$ which allows us to abbreviate the claimed inequality for an exponent pair $(p,r)$ as $R^{*}(p\to r) \lesssim 1$. 
Obstructions to inequalities $R^{*}(p\to r) \lesssim 1$ arise when the surface $P\subset F^d$ contains large affine subspaces. Indeed it was observed in \cite{MT04} that the presence of an affine subspace $H$ of dimension $k$ (that is $|H|=|F|^k$) implies the following necessary conditions for $R^{*}(p\to r) \lesssim 1$:
\begin{equation}\label{Necessary2}
r\geq \frac{2d}{d-1}  \quad \mbox{and} \quad r\geq\frac{p(d-k)}{(p-1)(d-1-k)}.\end{equation}
Informally speaking, the restriction conjecture for the paraboloid states that the necessary conditions that arise from subspaces in this way are the only obstructions to inequalities of this form.  The precise statement of the conjecture is complicated by the fact that the size of the largest affine subspace depends on the arithmetic of the field, in particular if there exists an element $i \in F$ such that $i^2 = -1$. For instance in the $d=3$ case, one can see that $P:=\{(x_1,x_2, x_1^2+ x_2^2) : x_1,x_2 \in F\}$ contains a line of the form $(t,it,0)$ when there exists an $i \in F$ such that $i^2=-1$. The size of the maximal affine subspaces contained within a $d$-dimensional paraboloid is well-understood, see \cite{Ko16, Le14}, and gives rise to the following end-point restriction conjectures:
\begin{enumerate} 
\item if $d\ge 2$ is even, then  $ R^*\left(\frac{2d^2}{d^2-d+2},~~ \frac{2d}{d-1}\right)\lesssim 1$;
\item if $d=4\ell-1$ for $\ell\in \mathbb N$, and $ -1\in  F$ is not a square number, then
$ R^*\left(\frac{2d^2+2d}{d^2+3},~~ \frac{2d}{d-1}\right)\lesssim 1$;
\item  If $d=4\ell+1$ for $\ell \in \mathbb N$, then $ R^*\left(\frac{2d}{d-1},~~ \frac{2d}{d-1}\right) \lesssim 1$; and
\item  if $d\ge 3$ is odd, and $-1\in  F$ is a square number, then
$R^*\left(\frac{2d}{d-1},~~ \frac{2d}{d-1}\right)\lesssim 1.$
\end{enumerate}
Recently the third author observed \cite{Le14} that one could potentially create more restrictive necessary conditions if one was able to construct certain highly overlapping packings of affine subspaces in $F^d$. In particular, the above restriction conjecture implies Dvir's theorem \cite{Dv09} that finite field Kakeya sets have full dimension. On the other hand, positive partial results have been obtained applying ideas from incidence geometry, Kakeya sets and operators, quadratic form theory, and analytic number theory. We refer the reader to \cite{Ko16} and \cite{Le14} for detailed summaries of the current best-known results in each of these cases.  

Here we will be interested in inequalities of the form $R^{*}(2\to r) \lesssim 1$ in even dimensions. It is elementary that $R^{*}(2\to 4) \lesssim 1$ for $d=2$, see \cite{MT04}, which is optimal. In higher dimensions the necessary conditions arising from affine subspaces suggest that one should have $R^{*}(2 \to r) \lesssim 1$ for $r \geq \frac{2d+4}{d}$. 

There has been a number of partial results to date in this direction. The finite field analog of the Stein-Tomas method, as developed by Mockenhaupt and Tao \cite{MT04}, gives $R^{*} (2 \to \frac{2d+2}{d-1} ) \lesssim 1$. This argument only uses the size and Fourier decay properties of the surface $P$, and applies in every dimension and is not sensitive to the arithmetic (such as if $-1$ is a square) of the finite field. In odd dimensions in which $-1$ is a square, this result gives the optimal $L^2$ restriction.

In even dimensions the Stein-Tomas result was improved by the first two authors \cite{IK09} in 2008 to $R^{*} (2 \to r ) \lesssim 1$ for $r > \frac{2d^2}{d^2-2d+2}$, with the endpoint $r=\frac{2d^2}{d^2-2d+2}$ obtained in \cite{LL10}. The second author \cite{Ko16} recently improved this to $r> \frac{6d+8}{3d-2}$ for $d\geq 6$, using, in part, ideas from \cite{Le13}. 

Here we further improve these results in $d\geq 4$ obtaining the optimal $L^2$ based results for even $d\geq 8$. In particular we prove

\begin{theorem}\label{newthm} Let $P\subset  F^d$ be the paraboloid. Then it follows that for every $\varepsilon >0$,
$$ R^*\left(2\to \frac{28}{9} \right) \lesssim 1 \quad\mbox{for}\,\,d=4;$$
$$ R^*\left(2\to \frac{8}{3} +\varepsilon\right) \lesssim 1 \quad\mbox{for}\,\,d=6; and$$
$$ R^*\left(2\to \frac{2d+4}{d} \right) \lesssim 1 \quad\mbox{for all even}\,\, d\ge 8.$$
\end{theorem}
The key new ingredient compared to prior work on this problem is improved estimates for the additive energy of subsets of the paraboloid of intermediate dimension. See Lemma \ref{lem:AE} below. 

\subsection{Discrete Fourier analysis} 
We give a brief overview of the notation and the basic Fourier analytic methods that are standard in this area. We refer the reader to \cite{Le14, MT04} for a more detailed exposition.

We will use $F$ to denote a finite field and write $F^d$ for the $d$-dimensional vector space over $F$ which we will endow with the counting measure $dx$. Given a function $f : F^d \to \mathbb C$ we will often consider levels where the value of the function is proportional in size to a given value. In particular we write $f(x) \sim \lambda$ to indicate that $\lambda/2 \leq |f(x)| \leq 2\lambda$. Given $x \in F^d$ we will let $\underline{x}$ denote the first $d-1$ coordinates of $x$ and $x_d$ the $d$-th coordinate or, in other words, $x=(\underline{x},x_d)$. This allows us to parameterize the $d$-dimensional paraboloid as $P:=\{(\underline{x},\underline{x}\cdot \underline{x}) : \underline{x} \in F^{d-1}\}$.

We let $e(\cdot)$ denote a nontrivial additive character on $F$ and recall the orthogonality relation 
$$|F|^{-d} \sum_{x\in F^d} e( \xi \cdot x) = \delta(\xi) =: \left\{ \begin{array}{ll} 0  \quad&\mbox{if}~~ \xi\neq (0,\dots,0),\\
1  \quad &\mbox{if} ~~\xi = (0,\dots,0). \end{array}\right.$$  
Given a complex-valued function $f : F^d \to \mathbb C$ we define its Fourier transform as
$$\widehat{f}(\xi) = \sum_{x \in F^d} f(x) e(-x\cdot \xi).$$
We will write $F^d_{*}$ for the dual space of $F^d$ and endow it with the normalized counting measure which assigns a mass of $|F|^{-d}$ to each point. The inverse Fourier transform of a function $g : F^d_{*} \to \mathbb C$ is then given by:
$$ g^\vee(x)= |F|^{-d}\sum_{\xi \in F^d} g(\xi)e(\xi\cdot x).$$
For a function $f: F^d\rightarrow \mathbb C$ we write $||f||_{L^p(F^d)}= \left(\sum_{x\in F^d} |f(x)|^p \right)^{1/p}$ and for $g: F^d_{*}\rightarrow \mathbb C$ we write $||g||_{L^p(F^d_{*})}= \left(|F|^{-d}\sum_{x\in F^d} |f(x)|^p \right)^{1/p}$. We will also  work with functions defined on a surface $S\subset F^{d}_{*}$ in which case we define the surface measure $d\sigma$ to the measure that assigns mass $|S|^{-1}$ to each element of $S$, so that $||g||_{L^p(F^d_{*},d\sigma)} = \left(|S|^{-1}\sum_{x\in S} |f(x)|^p \right)^{1/p}$. With these conventions, Parseval's theorem states $\|f^\vee\|_{L^2(F^d)}=\|f\|_{L^2(F^d_{*})}$ and
$$\left<f,g\right>_{F^{d}} = \left<\widehat{f},\widehat{g} \right>_{F^{d}_{*}} $$
where the first inner product is taken with respect to the counting measure and the second the normalized counting measure on $F^{d}_{*}$. From Parseval and duality it easily follows that the extension inequalities we seek to prove have equivalent restriction formulations. That is to say the following two inequalities are equivalent: 
\begin{equation}\label{eq:extenF}|| (fd\sigma)^{\vee}||_{L^{r}(F^d)} \leq R^{*}(p\to r) ||f||_{L^{p}(P,d\sigma)} \end{equation}
\begin{equation}\label{eq:restF} ||\widehat{g}||_{L^{p'}(P,d\sigma)} \leq R^{*}(p\to r) ||g||_{L^{r'}(F^d)} .
\end{equation}
As usual $r'$ and $p'$ denote the conjugate exponents of $r$ and $p$ respectively (i.e. $1=\frac{1}{p}+\frac{1}{p'}$). We define convolution on $F^n$ as $f * g (x) = \sum_{y \in F^{d}}f(y)g(x-y)$, and on $F^n_{*}$ as $f_{2} * g_{2} (\xi) = |F|^{-d}\sum_{\zeta \in F^{d}}f_{2}(\zeta)g_{2}(\xi-\zeta)$.
We will be interested in upper bounds on the quantity $||\widehat{g}||_{L^2(S,d\sigma)} := \left(|S|^{-1}\sum_{\xi \in S} |\widehat{g}(\xi)|^2 \right) $. From Parseval's theorem, we have that 
\begin{equation}\label{eq:dualL2}\left< \widehat{g},G \right> = \left<g, G*(d\sigma)^{\vee}\right>_{F^{d}}.\end{equation}
Letting $d\sigma$ denote the surface measure on the $d$-dimensional paraboloid and $\mathcal{G}(t) :=\sum_{\substack{s \in F \\ t \neq 0}}e(ts^2)$ denote the classical Gauss sum, an easy computation (see (18) in \cite{MT04}) gives us that
\begin{equation}\label{eq:evaulation}(d\sigma)^{\vee}(x) = |F|^{1-d} \sum_{\xi \in P} e(x\cdot \xi) 
= \left\{\begin{array}{ll} |F|^{1-d} e\left(\frac{\underline{x}\cdot\underline{x}}{-4x_d}\right) \mathcal{G}(x_d)^{d-1}\quad &\mbox{if} \quad x_d\ne 0\\

 0 \quad &\mbox{if} \quad x_d=0,~x\ne \mathbf{0} \\
                1 \quad &\mbox{if} \quad x=\mathbf{0},\end{array}\right. 
\end{equation}
where it is well-known (see (19) in \cite{MT04}) that $|\mathcal{G}(t)| = |F|^{\frac{1}{2}}$ for $t\neq 0$. It follows that $|(d\sigma)^{\vee}(x)| \leq |F|^{\frac{1-d}{2}}$ for $x\ne \mathbf{0}.$ Combining \eqref{eq:evaulation} with \eqref{eq:dualL2} one has that

\begin{equation}\label{eq:decayL1}||\widehat{g}||_{L^2(P,d\sigma)} = \left| \left<g, g*(d\sigma)^{\vee}\right>_{F^{d}} \right|^{\frac{1}{2}} \leq ||g||_{L^2(F^d)} + ||g||_{L^1(F^d)}|F|^{\frac{1-d}{4}}. \end{equation}

Interpolating this estimate with the following consequence of Parseval 
\begin{equation}\label{eq:bigL2}
\|\widehat{g}\|_{L^2(P, d\sigma)} \le |F|^{\frac{1}{2}} ||g||_{L^{2}(F^d)}
 \end{equation} yields the finite field Stein-Tomas estimate
$$ ||\widehat{g}||_{L^{2}(P,d\sigma)} \lesssim ||g||_{L^{\frac{2d+2}{d+3}}(F^d)} .$$
We refer the reader to \cite{MT04} for details. Another standard application (see \cite{Le14} Lemma 16) of \eqref{eq:decayL1} is the following epsilon removal lemma:
\begin{lemma}\label{lem:epsRem}Assume that for every $\epsilon>0$ one has a constant $c(\epsilon)$ such that the following inequality holds
$$ \|\widehat{g}\|_{L^2(P, d\sigma)} \lesssim c(\epsilon) |F|^{\epsilon} \|g\|_{L^{p_{0}}(F^d)}$$
then for every $p<p_0$, we have the estimate
$$ \|\widehat{g}\|_{L^2(P, d\sigma)} \lesssim \|g\|_{L^{p}(F^d)}.$$
\end{lemma}

\section{Restriction estimates from additive energy estimates}
We start by recalling the Mockenhaupt-Tao machine which allows one to relate the quantity $\|\widehat{g}\|_{L^2(P, d\sigma)}$ to the additive energy of the support of certain level sets of $g: F^d \rightarrow \mathbb C$. This same argument essentially appears, albeit in different notation, as Lemma 28 in \cite{Le14} or Lemma 3.6 in \cite{Ko16} (and is implicit in \cite{MT04}), but we sketch it below for the sake of completeness. Given a function $g:F^d \to \mathbb C$ we will denote its support by $G$. To help ease notation, we will also use $G$ to refer to the characteristic function of $G$. Moreover we define $\tilde{G_z} : F^{d-1} \rightarrow \{0,1\}$ by $\tilde{G_z}(\underline{x})= G(\underline{x},z)$. Finally we define $G_z: P \to \{0,1\}$ by $G_z(\underline{x},\underline{x}\cdot \underline{x})=\tilde{G_z}(\underline{x})$ (and $0$ for $x \notin P$). With this notation we have the following

\begin{lemma}\label{lem:MT0} Let $g : F^d \rightarrow \mathbb C$ such that $|g| \lesssim 1$ on its support. We then have
$$ \|\widehat{g}\|_{L^2(P, d\sigma)} \lesssim |supp(g)|^{\frac{1}{2}} + |supp(g)|^{\frac{3}{8}} |F|^{\frac{d-1}{4}}\left( \sum_{z \in F}||(G_z d\sigma)^\vee ||_{L^4(F^d)} \right)^{\frac{1}{2}} .$$
\end{lemma}
\begin{proof}(Sketch) From \eqref{eq:evaulation} we may write $(d\sigma)^{\vee}(x) = \delta(x) + K(x)$ with 
$$K(x) := |F|^{1-d} e\left(\frac{\underline{x}\cdot\underline{x}}{-4x_d}\right) \mathcal{G}(x_d)^{d-1}$$ 
for $x_d \neq 0$ and $0$ otherwise. Using \eqref{eq:dualL2} we have
$$ ||g||_{L^{2}(S,d\sigma)}^2 = \left<g, g*(d\sigma)^{\vee}\right>_{F^{d}}\le ||g||_{L^2(F^d)}^2 + ||g||_{L^{4/3}(F^d)} ||g * K||_{L^4(F^d)}.$$
Evaluation of the first two norms above, using the fact that $|g| \lesssim 1$, reduces to the lemma to the claim $||g * K||_{L^4(F^d)}\lesssim  |F|^{\frac{d-1}{2}} \times \sum_{z \in F}||(G_z d\sigma)^\vee ||_{L^4(F^d)}$. By the triangle inequality this reduces to the claim $||G_{z} * K||_{L^4(F^d)} \lesssim |F|^{\frac{d-1}{2}} ||(G_z d\sigma)^\vee ||_{L^4(F^d)}$. Since convolution commutes with translations, it suffices to establish this for $z=0$. We compute:
$$ \sum_{x \in F^d}|{G}_{z} * K(x)|^4 =  |F|^{2-2d}  \sum_{\underline{x} \in F^{d-1}}\sum_{\substack {x_d \in F \\ x_d \neq 0}} \left|\sum_{y \in F^{d-1}} \tilde{G}_{z}(y) e\left( (\underline{x}-y)\cdot(\underline{x}-y)/(-4x_d) \right)\right|^4.$$  
Expanding $(\underline{x}-y)\cdot(\underline{x}-y)/(-4x_d) = (\underline{x}\cdot \underline{x} -2\underline{x}\cdot y + y\cdot y)/(-4x_d)$ and performing the change of variables $t = 1/(-4 x_d)$ and $u = \underline{x}/(2x_d)$ we have
$$(\underline{x}-y)\cdot(\underline{x}-y)/(-4x_d) = u\cdot u/(4t) +u\cdot y +t y\cdot y.$$
With this change of variables, the left-hand side above is equal to
$$ |F|^{2-2d}  \sum_{u \in F^{d-1}}\sum_{\substack {t \in F \\ t \neq 0}} \left|e\left(\frac{u \cdot u}{4t} \right) \sum_{y \in F^{d-1}} \tilde{G}_{z}(y) e\left( u \cdot y + t y\cdot y  \right)\right|^4$$  
$$ =|F|^{2d-2}  \sum_{u \in F^{d-1}}\sum_{\substack {t \in F \\ t \neq 0}} \left| |P|^{-1}\sum_{y \in F^{d-1}} \tilde{G}_{z}(y) e\left( u \cdot y + t y\cdot y  \right)\right|^4$$
$$ \le  |F|^{2d-2}  \sum_{x \in F^{d}}| (G_{z}d\sigma)^{\vee}(x)|^{4}.$$ 
Taking fourth roots completes the proof.
\end{proof}

The usefulness of Lemma \ref{lem:MT0} follows from the fact that the $L^4$ norm of the extension operator applied to the slices of $g$, that is $||(G_z d\sigma)^\vee ||_{L^4(F^d)}$, has a combinatorial interpretation. In particular if $G_z$ is majorized by the characteristic function of a subset of $S\subseteq P$, then direct computation  (see Corollary 25 in \cite{Le14}) shows $||(1_S d\sigma)^\vee ||_{L^4(F^d)} = |F|^{\frac{4-3d}{4}} (\Lambda(S))^{\frac{1}{4}}$, where 
$$\Lambda(S):=\sum_{\substack{ a,b,c,d \in S \\ a+b=c+d}} 1$$  denotes the additive energy of the set of $S$.  The new ingredient in the current work is the following estimate for the additive energy of a subset of the paraboloid, which will be proven in the next section:

\begin{lemma}\label{lem:AE} Let $P\subset  F^d$ be the paraboloid.
If $d\ge 4$ is even  and $E\subseteq  P$, then we have
$$ \Lambda(E)=\sum_{\substack{x,y,z,w\in E \\ x+y=z+w}} 1 \lesssim  |F|^{-1}|E|^3 + |F|^{\frac{d-2}{2}} |E|^2.$$
\end{lemma}

Assuming this lemma and inserting it into the conclusion of Lemma \ref{lem:MT0} we obtain, identifying sets with their characteristic functions, that
\begin{lemma}\label{lem:MT} Let $g : F^d \rightarrow \mathbb C$ such that $g \sim 1$ on its support $G \subseteq F^d$. Then 
$$ \|\widehat{g}\|_{L^2(P, d\sigma)} \lesssim |G|^{\frac{1}{2}} + |G|^{\frac{3}{4}} |F|^{\frac{2-d}{8}} + |F|^{\frac{6-d}{16}}|G|^{\frac{5}{8}}.$$
\end{lemma}
\begin{proof}Starting from Lemma \ref{lem:MT0} and inserting the energy estimate from Lemma \ref{lem:AE} we have:
$$ \|\widehat{g}\|_{L^2(P, d\sigma)} \lesssim |G|^{\frac{1}{2}} + |G|^{\frac{3}{8}} |F|^{\frac{d-1}{4}}\, \left( \sum_{z \in F}||(G_z d\sigma)^\vee ||_{L^4(F^d)} \right)^{\frac{1}{2}}$$
$$ \lesssim |G|^{\frac{1}{2}}  + |G|^{\frac{3}{8}} |F|^{\frac{d-1}{4}} |F|^{\frac{4-3d}{8}} \left( |F|^{-\frac{1}{4}}\sum_{z \in F} |G_z|^{\frac{3}{4}} + |F|^{\frac{d-2}{8}} \sum_{z \in F}|G_z|^{\frac{1}{2}} \right)^{\frac{1}{2}} $$
$$ \lesssim |G|^{\frac{1}{2}}  + |G|^{\frac{3}{8}} |F|^{\frac{d-1}{4}} |F|^{\frac{4-3d}{8}} \left( |F|^{-\frac{1}{4}} |F|^{\frac{1}{4}} \left(\sum_{z \in F} |G_z|\right)^{\frac{3}{4}} + |F|^{\frac{d-2}{8}} |F|^{\frac{1}{2}} \left(\sum_{z \in F} |G_z|\right)^{\frac{1}{2}} \right)^{\frac{1}{2}}$$  
$$\lesssim |G|^{\frac{1}{2}} + |G|^{\frac{3}{4}} |F|^{\frac{2-d}{8}} + |F|^{\frac{6-d}{16}}|G|^{\frac{5}{8}}.$$
\end{proof}

Combining the dominate terms from Lemma \ref{lem:MT}, \eqref{eq:bigL2} and \eqref{eq:decayL1} we obtain the following:

\begin{lemma}\label{lem:CombEstimate}Let $d\ge 4$ is even. If $g :  F^d \rightarrow \mathbb C$ such that $|g| \lesssim 1$ on its support $G$, then
$$\|\widehat{g}\|_{L^2(P, d\sigma)} 
\lesssim \left\{  \begin{array}{ll}|F|^{\frac{1}{2}} |G|^{\frac{1}{2}} \quad&\mbox{for}~~|F|^{\frac{d+2}{2}} \le |G|\le |F|^{d}\\
|F|^{\frac{-d+6}{16}} |G|^{\frac{5}{8}} \quad&\mbox{for}~~ |F|^{\frac{3d+2}{6}} \le |G|\le |F|^{\frac{d+2}{2}}\\
|G|^{\frac{1}{2}} + |F|^{\frac{-d+1}{4}} |G|\quad&\mbox{for}~~ 1 \le |G|\le |F|^{\frac{3d+2}{6}}.\end{array}\right.
$$
\end{lemma}

\section{Proof of the Theorem \ref{newthm}}
We now deduce Theorem \ref{newthm} using Lemma \ref{lem:CombEstimate}. We proceed in the dual/restriction (see \eqref{eq:restF}) form. For even $d \geq 8$ it suffices to show $||\widehat{g}||_{L^2(P,d\sigma)} \lesssim 1$ for all $g: F^d \to \mathbb C$ with
\begin{equation}\label{eq:Lnorm}
\sum_{x \in F^d} |g(x)|^{\frac{2d+4}{d+4}} =1. 
\end{equation}
In the case $d=6$ it suffices to show $\|\widehat{g}\|_{L^2(P, d\sigma)} \lesssim \log |F|$ for all $g : F^6 \to \mathbb C$ such that $\sum_{x \in F^6} |g(x)|^{\frac{8}{5}} =1$ using the epsilon-removal/Lemma \ref{lem:epsRem}.  

Let $g_i$ denote the level set of $g$ where the function is of size $\sim 2^{-i}$. In other words, $g_i =  1_{\{x : g (x) \sim 2^{-i}\}}$. From \eqref{eq:Lnorm}, one has that $|supp(g_i)| \leq 2^{i \frac{2d+4}{d+4}}$. Our goal is to bound $\|\widehat{g}\|_{L^2(P, d\sigma)}$. We proceed as follows:

$$ \|\widehat{g}\|_{L^2(P, d\sigma)}  \lesssim \sum_{i=0}^{3 \log |F| } 2^{-i} \|\widehat{g_i}\|_{L^2(P, d\sigma)}$$
$$ \lesssim \sum_{\substack{0 \leq i \leq 3 \log |F|\\ 2^{i \frac{2d+4}{d+4}} \leq  |F|^{\frac{3d+2}{6}}}} 2^{-i} \|\widehat{g_i}\|_{L^2(P, d\sigma)} 
+ \sum_{\substack{0 \leq i \leq 3 \log |F|\\ |F|^\frac{3d+2}{6} \leq 2^{i \frac{2d+4}{d+4}} \leq  |F|^{\frac{d+2}{2}}}}2^{-i} \|\widehat{g_i}\|_{L^2(P, d\sigma)} 
+ \sum_{\substack{0 \leq i \leq 3 \log |F|\\ |F|^{\frac{d+2}{2}} \leq 2^{i \frac{2d+4}{d+4}} \leq  |F|^{d}}} 2^{-i} \|\widehat{g_i}\|_{L^2(P, d\sigma)} $$
$$ =I + II + III.$$
Applying Lemma \ref{lem:CombEstimate} with the estimate $|supp(g_i)| \leq 2^{i \frac{2d+4}{d+4}}$ we have that
$$I \leq  \sum_{\substack{0 \leq i \leq 3 \log |F|\\ 2^{i \frac{2d+4}{d+4}} \leq  |F|^{\frac{3d+2}{6}}}} 2^{-i} \|\widehat{g_i}\|_{L^2(P, d\sigma)} 
\lesssim \sum_{\substack{0 \leq i \leq 3 \log |F|\\ 2^{i \frac{2d+4}{d+4}} \leq  |F|^{\frac{3d+2}{6}}}}  2^{-i}2^{i \frac{2d+4}{2d+8}} + |F|^{\frac{1-d}{4} } 2^{-i} 2^{i \frac{2d+4}{d+4}}
\lesssim  1+\sum_{\substack{0 \leq i \leq 3 \log |F|\\ 2^{i \frac{2d+4}{d+4}}  \lesssim  |F|^{\frac{3d+2}{6}}}}  |F|^{\frac{6-d}{12d+24} }.$$
If $ d = 6$ we may conclude that $I \lesssim \log |F|$ and if $d \geq 8$ and is even then $I \lesssim 1$. Proceeding similarly with $II$ we have
$$II  \lesssim   \sum_{\substack{0 \leq i \leq 3 \log |F|\\ |F|^\frac{3d+2}{6} \leq 2^{i \frac{2d+4}{d+4}} \leq  |F|^{\frac{d+2}{2}}}}  |F|^{\frac{6-d}{16}} 2^{i \frac{d-6}{4d+16}}  \lesssim   |F|^{\frac{6-d}{16}}   \sum_{\substack{0 \leq i \leq 3 \log |F|\\ 0 \leq 2^{i } \leq |F|}}  2^{i \frac{d-6}{16}} . $$
From this we may conclude that for $d=6$ one has $II \lesssim \log |F|$ and if $d \geq 8$ is even then $II \lesssim 1$. Finally, for $III$ we have
$$ III  \lesssim |F|^{\frac{1}{2}}\sum_{\substack{0 \leq i \leq 3 \log |F|\\ |F|^{\frac{d+2}{2}} \leq 2^{i \frac{2d+4}{d+4}} \leq  |F|^{d}}} 2^{i \frac{2d+4}{2d+8} -i}  \lesssim |F|^{\frac{1}{2}}\sum_{\substack{0 \leq i \leq 3 \log |F|\\ |F|^{\frac{d+2}{2}} \leq 2^{i \frac{2d+4}{d+4}} \leq  |F|^{d}}} 2^{-i \frac{4}{2d+8} } \lesssim  
|F|^{\frac{1}{2}} |F|^{-\frac{1}{2}}.$$
So $III \lesssim 1$ for all $d$. This completes the proof for $d \geq 6$. We next consider the case $d=4$. Proceeding as above we set $p=\frac{28}{19},$ assume $\sum_{x \in F^d} |g(x)|^{p} =1,$ and note that $|supp(g_i)| \leq 2^{i p}$. We then may decompose 
$$ \|\widehat{g}\|_{L^2(P, d\sigma)}  \lesssim \sum_{i=0}^{3 \log |F| } 2^{-i} \|\widehat{g_i}\|_{L^2(P, d\sigma)}$$
$$ \lesssim \sum_{\substack{0 \leq i \leq 3 \log |F|\\ 2^{i p} \leq  |F|^{\frac{7}{3}}}} 2^{-i} \|\widehat{g_i}\|_{L^2(P, d\sigma)} 
+ \sum_{\substack{0 \leq i \leq 3 \log |F|\\ |F|^{\frac{7}{3}} \leq 2^{i p} \leq  |F|^{3}}}2^{-i} \|\widehat{g_i}\|_{L^2(P, d\sigma)} 
+ \sum_{\substack{0 \leq i \leq 3 \log |F|\\ |F|^{3} \leq 2^{i p} \leq  |F|^{4}}} 2^{-i} \|\widehat{g_i}\|_{L^2(P, d\sigma)} $$
$$ =I + II + III.$$
We estimate
$$ I \lesssim \sum_{\substack{0 \leq i \leq 3 \log |F|\\ 2^{i p} \leq  |F|^{\frac{7}{3}}}} \left( 2^{-i} |g_i|^{\frac{1}{2}} + 2^{-i}|F|^{-\frac{3}{4}}|g_i| \right)  \lesssim \sum_{\substack{0 \leq i \leq 3 \log |F|\\ 2^{i p} \leq  |F|^{\frac{7}{3}}}} \left( 2^{-i} 2^{i\frac{p}{2}} + |F|^{-\frac{3}{4}}2^{-i}  2^{i p} \right) \lesssim 1.$$
$$II \lesssim  \sum_{\substack{0 \leq i \leq 3 \log |F|\\ |F|^{\frac{7}{3}} \leq 2^{i p} \leq  |F|^{3}}}2^{-i} |F|^{\frac{1}{8}}|g_i|^{\frac{5}{8}} \lesssim 
\sum_{\substack{0 \leq i \leq 3 \log |F|\\ |F|^{\frac{7}{3}} \leq 2^{i p} \leq  |F|^{3}}}|F|^{\frac{1}{8}}2^{-i} 2^{\frac{5ip}{8}} \lesssim 1.$$
$$III \lesssim \sum_{\substack{0 \leq i \leq 3 \log |F|\\ |F|^{3} \leq 2^{i p} \leq  |F|^{4}}} |F|^{\frac{1}{2}} 2^{-i} |g_i|^{\frac{1}{2}}  \lesssim \sum_{\substack{0 \leq i \leq 3 \log |F|\\ |F|^{3} \leq 2^{i p} \leq  |F|^{4}}} |F|^{\frac{1}{2}} 2^{-i} 2^{\frac{ip}{2}} \lesssim 1.  $$
This completes the proof.

\section{Additive Energy Estimates}
In this section, we complete the proof of the theorem by verifying the additive energy estimate stated above, namely 
\begin{lemma}\label{lem:AE} Let $P\subset  F^d$ be the $d$-dimensional paraboloid.
If $d\ge 4$ is even  and $E\subseteq  P$, then we have
$$ \Lambda(E)=\sum_{\substack{x,y,z,w\in E \\ x+y=z+w}} 1 \lesssim  |F|^{-1}|E|^3 + |F|^{\frac{d-2}{2}} |E|^2.$$
\end{lemma}

The proof of this result will be Fourier analytic and has similarities to arguments used previously in \cite{IK09, IK10, IR07}. As above, we will parameterize elements of $x \in P$ as $x=(\underline{x},\underline{x} \cdot \underline{x})$ and use $\underline{x}$ to denote the projection of $x$ onto $F^{d-1}$. Similarly if $E \subseteq P$ we will denote the set obtained by projecting the elements of $E$ onto $F^{d-1}$ as $\underline{E}$.

We start by decomposing the sum defining the additive energy into two components. These components will be distinguished based on the relationship of certain pairs of points to $j$-spheres. We define a $j$-sphere $S_j$ in $ F^{d-1}$ as
$$S_j := \{x \in F^{d-1}:x\cdot x = j\}.$$
We will discuss certain properties of these spheres below. We may now start the proof of Lemma \ref{lem:AE}. Let $E\subseteq P \subset F^d.$

\begin{align*} \Lambda(E)&=\sum_{\substack{x,y,z,w\in E\\x+y=z+w}}1
\le \sum_{\substack{x,y,z\in E\\ x-z+y\in P}}1\\
&= \sum_{\substack{x,y,z\in E\\ x-z+y\in P,\\ \underline{z}-\underline{y}\in S_0}}1+ \sum_{\substack{x,y,z\in E\\ x-z+y\in P,\\ \underline{z} -\underline{y} \notin S_0}}1 :={M_1} + {M_2}.\end{align*}

The proof will be completed by bounding $M_1$ and $M_2$. These will be considered in the following two subsections, respectively.

\subsection{Bounding $M_1$}
We start by estimating
$$M_1 := \sum_{\substack{x,y,z\in E\\ x-z+y\in P,\\ \underline{z}-\underline{y}\in S_0}}1  \leq \sum_{\substack{x,y,z\in E \\ \underline{z}-\underline{y}\in S_0}}1= |E| \sum_{\substack{\underline{y},\underline{z} \in F^{d-1}}} \underline{E}(\underline{y})\underline{E}(\underline{z})S_{0}(\underline{z}-\underline{y}).$$
We will proceed by expanding the characteristic function of $S_0$ in a Fourier series 
$$S_{0}(\alpha) = \sum_{m \in F^{d-1}} S_{0}^{\vee}(m) e(-m\cdot\alpha).$$
This gives us that 
$$M_1 \leq |E| \sum_{\substack{\underline{y},\underline{z} \in F^{d-1}}} \underline{E}(\underline{y})\underline{E}(\underline{z})\sum_{m \in F^{d-1}} S_{0}^{\vee}(m) e(-m\cdot(\underline{z}-\underline{y}))$$
$$= |E| |F|^{2d-2} \sum_{m \in F^{d-1}} S_{0}^{\vee}(m)|\underline{E}^{\vee}(m)|^2$$
\begin{equation}\label{eq:M1}= |E| |F|^{2d-2}S_{0}^{\vee}(\mathbf{0})|\underline{E}^{\vee}(\mathbf{0})|^2   +|E| |F|^{2d-2} \sum_{\substack{m \in F^{d-1} \\ m \neq \mathbf{0}}} S_{0}^{\vee}(m)|\underline{E}^{\vee}(m)|^2.\end{equation}
The function $S_{0}^{\vee}(m)$ can be explicitly evaluated in terms of Gauss and Kloosterman sums. Denoting $F_{\times}$ to be the set of non-zero elements of  $F$, $e(\cdot)$ a nontrivial additive character, and $\eta$ the multiplicative quadratic character on $F_\times$. We then define the Gauss sum $\mathcal{G}(\eta):= \sum_{s \in F_{\times}} \eta(s)e(s)$ and twisted Kloosterman sum $TK(a,b):= \sum_{s \in F_{\times}} \eta(s) e(as+bs^{-1})$. It is well known (see, for example, \cite{LN97, IK04}) that: 
\begin{equation}\label{AbsExp}
|\mathcal{G}(\eta)|=|F|^{\frac{1}{2}} \quad \mbox{and}\quad |TK(a,b)|\le 2 |F|^{\frac{1}{2}}.
\end{equation}
Moreover, the inverse Fourier transform of the $n$-dimensional $j$-sphere, $S_j$, is explicitly calculated (see, for example, Lemma 4 in \cite{IK10}), as follows:
$$S_j^\vee(\alpha):=\frac{1}{|F|^{n}}\sum_{\xi\in S_j} e(\alpha\cdot \xi)$$
\begin{equation} =|F|^{-1} \delta(\alpha) + |F|^{-n-1}\,\eta^{n}(-1)\, \mathcal{G}^{n}(\eta) \sum_{r \in F_\times} 
\eta^{n}(r)\,e\Big(jr+ \frac{\alpha \cdot \alpha}{4r}\Big).
\end{equation}
Estimating the inner exponential sums using \eqref{AbsExp} gives the following.
\begin{lemma}\label{sphereCoBound} Let $S_0= \{x \in F^n:x\cdot x =0 \}$ be the $0$-sphere in $F^n$. If $n\ge 3$ is odd, then we have
$$\left|S_0^\vee(\alpha)\right|\lesssim |F|^{-\frac{(n+1)}{2}} \quad \mbox{for}~~\alpha \neq \mathbf{0} $$
and $\left|S_0^\vee(\mathbf{0})\right|\sim |F|^{-1}.$ If $n\ge 4$ is even, then we have
$$\left|S_0^\vee(\alpha)\right|\lesssim |F|^{-\frac{n}{2}} \quad \mbox{for}~~\alpha \neq \mathbf{0}$$
and $\left|S_0^\vee(\mathbf{0})\right|\sim |F|^{-1}.$
\end{lemma}
Applying Lemma \ref{sphereCoBound} to \eqref{eq:M1} for odd dimensions (since $d$ is even) and the identity $\sum_{m \in F^{d-1}} |\underline{E}^\vee(m) |^2 = |F|^{-(d-1)}|E|$  we have that
$$M_1 \lesssim |E| |F|^{2d-2}\times |F|^{-1} \times \left(|F|^{-(d-1)} |E|\right)^{2} + |E|  |F|^{2d-2} \times |F|^{-d/2} \times|F|^{-(d-1)} |E|$$
$$\lesssim |F|^{-1}|E|^3 + |F|^{\frac{d-2}{2}}|E|^2.$$
This completes the estimate of $M_1$.

\subsection{Bounding $M_2$}
We start by estimating
\begin{equation}\label{eq:M2}M_2:=\sum_{\substack{x,y,z\in E\\ x-z+y\in P,\\ \underline{z}- \underline{y} \notin S_0}} 1= \sum_{y\in E} \left( \sum_{\substack{x,z\in E\\ x-z+y\in P,\\ \underline{z} -\underline{y} \notin S_0}} 1\right).
\end{equation}
Next we will perform a Galilean transformation which will allow us to simplify the inner sum. Given $y=(\underline{y}, y_d)\in P$ we define the Galilean transformation $G_{-y}:P \to P$ by
$$ G_{-y}(\underline{z}, z_d) = (\underline{z}-\underline{y},\, z_d-2\underline{z}\cdot \underline{y}+ \underline{y}\cdot \underline{y}).$$
The following two properties of $G_{-y}$ make it useful.
\begin{lemma}\label{lem:Gal}Let $x,z,y \in P$ we have that
$$x-z+y\in P \quad\mbox{if and only if}\quad G_{-y}(x)-G_{-y}(z)\in P.$$
$$\underline{z}-\underline{y} \notin S_0 \quad\mbox{if and only if}\quad G_{-y}(z) \notin S_0\times\{0\} .$$
\end{lemma}
\begin{proof}The first claim can be verified by directly computing and comparing the algebraic equations implied by each side of the equivalence. This is done in detail in the proof of Lemma 44 of \cite{Le14}. The second claim follows immediately from the definition of $G_{-y}(z)$.
\end{proof}
We will also need the following elementary fact.
\begin{lemma}\label{lem:Fnorms}Let $S_0$ denote the $n$-dimensional $0$-sphere and $x, y \in F^{n}\setminus S_0$. Then $x = y$   if and only if  $\frac{x}{x\cdot x} = \frac{y}{y\cdot y}.$
\end{lemma}
\begin{proof}Clearly $x = y$ implies $\frac{x}{x\cdot x} = \frac{y}{y\cdot y}$. Conversely, suppose that $\frac{x}{x\cdot x} = \frac{y}{y\cdot y}$ for $x, y \in F^{n}\setminus S_0$. Letting $t:=\frac{y\cdot y}{x\cdot x} \neq 0$ we have that $y=tx$. This implies 
$$\frac{x}{x\cdot x} = \frac{y}{y\cdot y}= \frac{tx}{t^2 (x\cdot x)}$$ 
which, in turn, implies $t=1$ and thus $x=y$ which completes the proof.
\end{proof}

Applying the Lemma \ref{lem:Gal} to \eqref{eq:M2} we see that
$$M_2 = \sum_{y\in E} \left( \sum_{\substack{x,z\in E\\ x-z+y\in P,\\ \underline{z} -\underline{y}\notin S_0}}1\right) 
= \sum_{y\in E} \left( \sum_{\substack{G_{-y}(x),\,G_{-y}(z)\in G_{-y}(E)\\ G_{-y}(x)-G_{-y}(z)\in P,\\ G_{-y}(z) \notin S_0\times \{0\}}} 1\right).$$
For each fixed $y\in E$, let $G_{-y}(E)=E_y$ and $E_y\setminus (S_0\times \{0\}) := E_y^{*}.$ It follows that
$$ M_2=\sum_{y\in E} \left( \sum_{\substack{a\in E_y,\, b\in E_y^{*}\\ a-b\in P}}1\right) \le |E| \max_{y\in E}\left( \sum_{\substack{a\in E_y,\, b\in E_y^{*}\\ a-b\in P}}1\right). $$

The desired bound on $M_2$ is now reduced to the following lemma.

\begin{lemma}\label{Mlem1}
Let $S_0$ denote the $0$-sphere in $F^{d-1}$. If $A\subseteq P $ and $B\subseteq P\setminus (S_0\times \{0\}),$  we have
$$\sum_{\substack{x\in A, y\in B\\ x-y\in P}}1 \le |F|^{-1} |A||B| + |F|^{\frac{d-2}{2}} |A|^{\frac{1}{2}}  |B|^{\frac{1}{2}}.$$ \end{lemma}
\begin{proof}
Observing that $x-y\in P$ for $x\in A\subseteq P$ and $y\in B\subseteq P\setminus (S_0\times \{0\})$ if and only if $ \underline{x} \cdot \underline{y} = \underline{y} \cdot \underline{y} \ne 0$. It follows that
$$ \sum_{\substack{x\in A, y\in B\\ x-y\in P}}1 = \sum_{\substack{\underline{x} \in \underline{A}, \underline{y} \in \underline{B}\\\underline{x} \cdot \underline{y} = \underline{y}\cdot \underline{y} \ne 0 }}1=\sum_{\substack{\underline{x} \in \underline{A}, \underline{y}\in \underline{B}\\\underline{x} \cdot \frac{\underline{y}}{\underline{y}\cdot \underline{y}} = 1}}1.$$
Using Lemma \ref{lem:Fnorms}, we can replace $\underline{B}$ with $\underline{B}^\prime:=\{\frac{\underline{y}}{\underline{y}\cdot \underline{y}} : y \in B\}$ and  we can write
$$ \sum_{\substack{x\in A, y\in B\\ x-y\in P}}1 
=\sum_{\substack{\underline{x} \in \underline{A} , \underline{y} \in \underline{B} \\ \underline{x} \cdot \frac{\underline{y}}{\underline{y}\cdot \underline{y}} = 1}}1 
=  \sum_{\substack{a\in \underline{A}, b\in \underline{B}^\prime \\a \cdot b= 1}}1.$$
We will estimate this quantity using Fourier analysis. Note that $|B|=|\underline{B}|=|\underline{B}^\prime|$. By the orthogonality relation of the nontrivial additive character $e(\cdot)$ of $F$,  it follows that
\begin{align*}\sum_{\substack{a\in \underline{A}, b\in \underline{B}^\prime \\a \cdot b= 1}}1 
&= \sum_{a\in \underline{A}, b\in \underline{B}^\prime } |F|^{-1} \sum_{s\in F} e(s(a\cdot b-1))\\
&=|F|^{-1}|A||B| + |F|^{-1} \sum_{\substack{a\in \underline{A}, b\in \underline{B}^\prime\\s\ne 0}} e(s(a\cdot b -1))\\
&:= |F|^{-1}|A||B| + R. \end{align*}
Estimating $R$ by applying the Cauchy-Schwarz inequality in the variable $b$, and
completing the sum over $b \in \underline{B}^{\prime}$ to all of $F^{d-1}$, we see that
\begin{align*}
|R|^2 &\le |F|^{-2} |B| \sum_{b\in F^{d-1}} \left| \sum_{a\in \underline{A}, s\ne 0} e(s(a\cdot b-1))\right|^2\\
& =|F|^{-2} |B| \sum_{a, a^\prime\in \underline{A},\, s,s^\prime \ne 0} 
e(s^\prime-s) \left(\sum_{b\in F^{d-1}} e( (sa-s^\prime a^\prime)\cdot b)\right).
\end{align*}
By the orthogonality relation of $e(\cdot)$, we have
\begin{align*}
|R|^2 &\le |F|^{d-3} \,|B| \sum_{\substack{a, a^\prime\in \underline{A}, \,s, s^\prime \ne 0\\sa=s^\prime a^\prime}} e(s^\prime-s)\\
&= |F|^{d-3}\, |B| \sum_{a\in \underline{A},\, s\ne 0} 1 +|F|^{d-3} \,|B|
\sum_{\substack{a, a^\prime\in \underline{A}, \,s, s^\prime \ne 0\\s\ne s^\prime,\, sa=s^\prime a^\prime}} e\left(s^\prime \left(1- \frac{s}{s^\prime}\right)\right)\\
&\le |F|^{d-2}|A||B| +|F|^{d-3} \,|B|
\sum_{\substack{a, a^\prime\in \underline{A}, \,s, s^\prime \ne 0\\s\ne s^\prime,\, sa=s^\prime a^\prime}} e\left(s^\prime \left(1- \frac{s}{s^\prime}\right)\right).
\end{align*}
Letting $u=s^\prime$ and $v=\frac{s}{s^\prime}$, it follows that
\begin{align*} |R|^2 &\le |F|^{d-2}|A||B| +|F|^{d-3} \,|B|
\sum_{\substack{a, a^\prime\in F^{d-1}, \,u\ne 0, v \ne 0,1\\va= a^\prime}} \underline{A}(a) \,\underline{A}(a^\prime)\, e(u(1-v))\\
&= |F|^{d-2}|A||B| -|F|^{d-3}\,|B| \sum_{a\in F^{d-1}, \,v\ne 0,1} \underline{A}(a)\,\underline{A}(va) \le |F|^{d-2}|A||B|.
\end{align*}
Thus, $|R|\le |F|^{\frac{d-2}{2}} |A|^{\frac{1}{2}}  |B|^{\frac{1}{2}}$ which completes the proof.
\end{proof}

 \bibliographystyle{amsplain}

\begin{thebibliography}{10}

\bibitem{Bo91} J.~ Bourgain, \emph{Besicovitch-type maximal operators and applications to Fourier analysis},
Geom. and Funct. Anal. \textbf{22} (1991), 147-187.

\bibitem{BG} J. Bourgain and L. Guth, \emph{Bounds on oscillatory integral operators based on multilinear estimates}, Geometric And Functional Analysis \textbf{21} (2011),  6, 1239-1295.

\bibitem{BCT06} J. Bennett,  A.Carbery, and T. Tao, \emph{On the multilinear restriction and Kakeya conjecture,} Acta Math., \textbf{196} (2006), 261-302.

\bibitem{Dv09} Z. Dvir,  \emph{On the size of Kakeya sets in finite fields}, J. Amer. Math. Soc. \textbf{22} (2009), 1093-1097.
\bibitem{Gu15} L. Guth, \emph{A restriction estimate using polynomial partitioning},  J. Amer. Math. Soc. \textbf{29}  (2016),  no. 2, 371-413. 
\bibitem{Gu16} L. Guth, \emph{Restriction estimates using polynomial partitioning II}, preprint (2016), arXiv:1603.04250.
\bibitem{IK08} A. Iosevich and D. Koh, \emph{Extension theorems for the Fourier transform associated with non-degenerate quadratic surfaces in vector spaces over finite fields}, Illinois J. of Mathematics, \textbf{52} (2008), no.2,  611-628.
\bibitem{IK09} A. Iosevich and D. Koh, \emph{Extension theorems for paraboloids in the finite field setting}, Math. Z. {\bf 266} (2010),  471-487.
\bibitem{IK10} A. Iosevich and D. Koh, \emph{Extension theorems for spheres in the finite field setting}, Forum. Math. {\bf 22} (2010), no.3, 457-483.
\bibitem{IK04} H. Iwaniec and E. Kowalski, \emph{Analytic number theory}, AMS Colloquium Publications, \textbf{53}, (2004).
\bibitem{IR07} A.~ Iosevich and M. ~Rudnev, \emph{Erd\"{o}s distance problem in vector spaces over finite fields}, Trans. Amer. Math. Soc. \textbf{359} (2007), 6127-6142.
\bibitem{Ko16} D. Koh, \emph{Conjecture and improved extension theorems for paraboloids in the finite field setting}, preprint (2016), arXiv:1603.06512.
\bibitem{Le13}  M.~Lewko, \emph{New restriction estimates for the 3-d paraboloid over finite fields},   Adv. Math. {\bf 270} (2015), no. 1, 457-479.
\bibitem{Le14}  M.~Lewko, \emph{Finite field restriction estimates based on Kakeya maximal operator estimates},  JEMS, to appear.
\bibitem{LL10} A.~ Lewko and M.~Lewko, \emph{Endpoint restriction estimates for the paraboloid over finite fields}, Proc. Amer. Math. Soc. \textbf{140} (2012), 2013-2028.
\bibitem{LN97} R. Lidl and H. Niederreiter, \emph{ Finite fields,} Cambridge University Press, (1997).
\bibitem{MT04} G. Mockenhaupt, and T. Tao, \emph{Restriction and Kakeya phenomena for finite fields}, Duke Math. J. \textbf{121} (2004), 1, 35-74.
\bibitem{St78} E.M. Stein, \emph{Some problems in harmonic analysis}, Harmonic analysis in Euclidean spaces (Proc. Sympos. Pure Math., Williams Coll., Williamstown, Mass., 1978), Part 1, pp. 3-20.
\bibitem{SZ16} A. Stevens and F. Zeeuw, \emph{An improved point-line incidence bound over arbitrary fields}, preprint (2016), arXiv:1609.06284.
\bibitem{Ta03}T. Tao, \emph{A sharp bilinear restriction estimate for paraboloids}, Geom. Funct. Anal. \textbf{13} (2003), 1359-1384.
\bibitem{Ta04} T. Tao, \emph{Some recent progress on the restriction conjecture}, Fourier analysis and convexity,  217-243, Appl. Numer. Harmon. Anal., Birkh\"auser Boston, Boston, MA, (2004).


\end{thebibliography}

\end{document}